\documentclass[twoside,reqno,bibtotoc]{elsarticle}
\usepackage[utf8]{inputenc}
\usepackage{xcolor}
\usepackage{pdfcolmk}
\usepackage{prettyref}
\usepackage{textcomp}
\usepackage{mathrsfs}
\usepackage{mathtools}
\usepackage{amsmath}
\usepackage{amsthm}
\usepackage{amssymb}
\usepackage{stmaryrd}
\PassOptionsToPackage{normalem}{ulem}
\usepackage{ulem}
\usepackage{microtype}
\usepackage[unicode=true,
 bookmarks=false,
 breaklinks=false,pdfborder={0 0 1},backref=false,colorlinks=true]
 {hyperref}
\hypersetup{
 pdfstartview=FitV,linkcolor=black,citecolor=black}

\makeatletter

\providecolor{lyxadded}{rgb}{0,0,1}
\providecolor{lyxdeleted}{rgb}{1,0,0}

\DeclareRobustCommand{\lyxsout}[1]{\ifx\\#1\else\sout{#1}\fi}

\numberwithin{equation}{section}
\numberwithin{figure}{section}
\theoremstyle{plain}
\newtheorem{thm}{\protect\theoremname}[section]
\theoremstyle{remark}
\newtheorem{rem}[thm]{\protect\remarkname}
\theoremstyle{plain}
\newtheorem{lem}[thm]{\protect\lemmaname}
\theoremstyle{plain}
\newtheorem{prop}[thm]{\protect\propositionname}
\theoremstyle{definition}
\newtheorem{defn}[thm]{\protect\definitionname}
\theoremstyle{definition}
\newtheorem{example}[thm]{\protect\examplename}

\journal{~}

\usepackage{bbm}
\usepackage{lineno}
\usepackage{pgfplots}
\pgfplotsset{compat=1.15}
\usepackage{mathrsfs}
\usetikzlibrary{arrows}

\usepackage{txfonts}
\usepackage{textcomp}
\usepackage{xcolor}
\usepackage{dsfont}
\usepackage{mathrsfs}
\usepackage{graphicx}
\usepackage{tikz}

\renewcommand{\rho}{\varrho}

\newcommand{\reftext}{}

\DeclareMathOperator{\dom}{dom}

\DeclareMathOperator{\spann}{span}
\DeclareMathOperator{\supp}{supp}
\DeclareMathOperator{\card}{card}

\renewcommand{\phi}{\varphi}
\renewcommand{\epsilon}{\varepsilon}

\newrefformat{prop}{Proposition~\ref{#1}}
\newrefformat{fact}{Fact~\ref{#1}}
\newrefformat{exa}{Example~\ref{#1}}
\newrefformat{cor}{Corollary~\ref{#1}}
\newrefformat{assu}{Assumption~\ref{#1}}
\newrefformat{subsec}{Section~\ref{#1}}
\newrefformat{sec}{Section~\ref{#1}}
\newrefformat{def}{Definition~\ref{#1}}

\makeatother

\providecommand{\definitionname}{Definition}
\providecommand{\examplename}{Example}
\providecommand{\lemmaname}{Lemma}
\providecommand{\propositionname}{Proposition}
\providecommand{\remarkname}{Remark}
\providecommand{\theoremname}{Theorem}

\begin{document}

\begin{frontmatter}{}

\title{Approximation order of Kolmogorov diameters via $L^{q}$-spectra
and applications to polyharmonic operators}

\author{Marc~Kesseböhmer\corref{cor1}}

\ead{mhk@uni-bremen.de}

\author{Aljoscha~Niemann\corref{cor1}}

\ead{niemann1@uni-bremen.de}

\address{Fachbereich 3 -- Mathematik und Informatik, University of Bremen,
Bibliothekstr. 1, 28359 Bremen, Germany}

\cortext[cor1]{Corresponding authors}
\begin{abstract}
We establish a connection between the $L^{q}$-spectrum of a Borel measure
$\nu $ on the $m$-dimensional unit cube and the approximation order of
Kolmogorov diameters of the unit sphere with respect to Sobolev norms in
$L_{\nu }^{p}$. This leads to improvements of classical results of Borzov
and Birman/Solomjak for a broad class of singular measures. As an application,
we consider spectral asymptotics of polyharmonic operators and obtain improved
upper bounds of the decay rate of their eigenvalues. For measures with non-trivial absolutely continuous parts as well as for self-similar measures the exact approximation orders are stated.\end{abstract}

\end{frontmatter}{}

\tableofcontents{}

\section{Introduction and statement of main results}
\label{sec1}

Let us start with some basic notations. Let $\mathbb{R}^{m}$ denote the
$m$-dimensional euclidean space, $m\in \mathbb N$. For a multi-index
$k :=  \left (k_{1},\dots ,k_{m}\right )\in \mathbb N^{m}$ we define
$|k| :=  \sum _{i=1}^{m}k_{i}$ and
$x^{k} :=  \prod _{i=1}^{m}x_{i}^{k_{i}}$ and for a bounded open subset
$\Omega \subset \mathbb{R}^{m}$ and $p\geq 1$, we let
$L^{p}(\Omega )$ denote the set of real-valued $p$-integrable functions
on $\Omega $ with respect to the Lebesgue measure $\Lambda $ restricted
to $\Omega $. The Sobolev space $W_{p}^{\ell}(\Omega )$ (see e.g.
\citep{MR1125990} and \citep{1961RuMa}) is defined to be the set of all
functions $f\in L^{p}(\Omega )$ for which the weak derivatives up to order
$\ell \in \mathbb N$ lie in $L^{p}(\Omega )$ and
\begin{equation*}
\left \Vert f\right \Vert _{W_{p}^{\ell}(\Omega )} :=  \left
\Vert f\right \Vert _{L^{p}(\Omega )}+\left \Vert f\right \Vert _{L^{
\ell ,p}(\Omega )}<\infty ,
\end{equation*}
where we set
$\left \Vert f\right \Vert _{L^{\ell ,p}(\Omega )} :=  \left (
\int _{\Omega}\left |\nabla _{\ell}f\right |^{p}\;\mathrm d\Lambda
\right )^{1/p}$ with
$\left |\nabla _{\ell}f\right | :=  \left (\sum _{|k|=\ell}
\left |D^{k}f\right |^{2}\right )^{1/2}$ and
$D^{k}f :=  \partial ^{|k|}/\left (\partial _{x_{1}}^{k_{1}}
\cdot \cdot \cdot \partial _{x_{\ell}}^{k_{\ell}}\right )f$. We let
$W_{0,p}^{\ell}(\text{$\Omega$)}$ denote the completion of
$\mathcal{C}_{c}^{\infty}(\Omega )$ with respect to
$\left \Vert \,\cdot \,\right \Vert _{W_{p}^{\ell}(\Omega )}$, where
$\mathcal{C}_{c}^{\infty}(\Omega )$ denotes the set of all infinitely differentiable
functions with compact support in $\Omega $. For any half-open cube
$Q\subset \mathbb{R}^{m}$, which---throughout the paper---are assumed to
have edges parallel to the coordinate axes, we have by definition of the
weak derivatives that $W_{p}^{\ell}(Q)=W_{p}^{\ell}(\mathring{Q})$, where
$\mathring{Q}$ denotes the interior of $Q$. Note that for the space
$W_{0,p}^{\ell}(Q)$ an equivalent norm is given by
$\left \Vert \,\cdot \,\right \Vert _{L^{\ell ,p}(Q)}$. If
$\ell p/m>1$, then $W_{p}^{\ell}\left (Q\right )$ is compactly embedded
into
$\left (\mathcal{C}(\overline{Q}),\left \Vert \,\cdot \,\right \Vert _{
\mathcal{C}(\overline{Q})}\right )$, with
$\left \Vert \,\cdot \,\right \Vert _{\mathcal{C}(\overline{Q})}$ denoting
the uniform norm, and therefore we will always pick a continuous representative
of $W_{p}^{\ell}\left (Q\right )$. For the set of continuous function from
$A\subset \mathbb{R}^{m}$ to $\mathbb{R}$ we write
$\mathcal{C}\left (A\right )$.

For a normed vector space
$\left (V,\left \Vert \,\cdot \,\right \Vert _{V}\right )$ and a subset
$K\subset V$, the \textit{Kolmogorov $n$-diameter} (or $n$\emph{-widths})\emph{
of $K$ in $V$}, $n\in \mathbb N$, is given by
\begin{align*}
d_{n}\left (K,V\right ) &  :=  \inf \left \{ \sup _{x\in K}\inf _{y
\in V_{n}}\left \Vert x-y\right \Vert _{V}:V_{n}
\text{ is $n$-dimensional subspace of }V\right \} .
\end{align*}
If $K$ is pre-compact, then the $n$-diameters converge to zero and one
could say that the $n$-diameter $d_{n}\left (K,V\right )$ measures the
extend to which $K$ can be approximated by $n$-dimensional subspaces of
$V$. We call the value
\begin{equation*}
\overline{\mathbf{ord}}\left (K,V\right ) :=  \limsup _{n\to
\infty}\frac{\log \left (d_{n}\left (K,V\right )\right )}{\log n}
\end{equation*}
the \emph{upper} \emph{approximation order of $K$ in $V$}. If the upper approximation
order coincides with the \emph{lower approximation order}
$\mathscr{\underline{\mathbf{ord}}}\left (K,V\right )$ defined by replacing
the limit superior with the limit inferior in the above definition, we
call the common value $\mathscr{\mathbf{ord}}\left (K,V\right )$ the
\emph{approximation order}. See \citep{MR774404} for further details on
this topic. For the \emph{unit sphere} in $V$ we write
$\mathscr{S}V :=  \left \{ f\in V:\left \Vert f\right \Vert _{V}=1
\right \} $. In the following we will concentrate on the particular choice
$V=L_{\nu}^{q}\left (\textbf{Q}\right )$ for a Borel measure $\nu $ on the
half-open unit cube $\textbf{Q} :=  \left (0,1\right ]^{m}$ and
$K\in \left \{ \mathscr{S}W_{p}^{\ell}\left (\textbf{Q}\right ),
\mathscr{S}W_{0,p}^{\ell}\left (\textbf{Q}\right )\right \} $,
$q\geq p>1$. Throughout, we will assume that
%
\begin{equation}
\varrho  :=  q\left (\ell -m/p\right )>0\;\:\text{and }\:q\geq p>1.
\label{eq:StandingAssumption}
\end{equation}
Under this condition, using the Landau symbols, it has been shown in
\citep{MR0217487} that
%
\begin{equation}
d_{n}\left (\mathscr{S}W_{p}^{\ell},L_{\nu}^{q}\right )=O\left (n^{-
\left (\ell /m-1/p+1/q\right )}\right )
\label{eq:BirmanSolomanjakOld}
\end{equation}
and in the case that $\nu $ is a singular measure with respect to the Lebesgue
measure, we know from \citep{Borzov1971} that even
\begin{equation*}
d_{n}\left (\mathscr{S}W_{p}^{\ell},L_{\nu}^{q}\right )=o\left (n^{-
\left (\ell /m-1/p+1/q\right )}\right ).
\end{equation*}
In this paper we want to address the question to what extent these estimates
can be effectively improved for arbitrary Borel measures on
$\textbf{Q}$. We will see how our main result can be obtained from auxiliary
measure-geometric quantities involving the $L^{q}$-spectrum of
$\nu $ combined with some ideas from \citep{MR0217487} dealing with piecewise
polynomial approximation in $L_{\nu}^{q}$ of elements in
$W_{p}^{\ell}(\textbf{Q})$ (see Section~\ref{subsec:Theorems-on-approximation}).
For $n\in \mathbb N$, we set
\begin{equation*}
\mathcal{D}_{n} :=  \left \{ Q=\prod _{k=1}^{m}\left (l_{k}2^{-n},(l_{k}+1)2^{-n}
\right ]:(l_{k})_{k=1,\dots ,m}\in \mathbb{Z}^{m},\nu \left (Q\right )>0
\right \} ,\mathcal{D} :=  \bigcup _{n\in \mathbb N}\mathcal{D}_{n}
\end{equation*}
and the \emph{$L^{q}$-spectrum} of $\nu $ is given, for
$s\in \mathbb{R}$, by
\begin{equation*}
\beta _{\nu}(s) :=  \limsup _{n\rightarrow \infty}\beta _{n}^{
\nu}(s)\,\text{\,with\,\, }\beta _{n}^{\nu}\left (s\right ) :=
\log \left (\sum _{C\in \mathcal{D}_{n}}\nu (C)^{s}\right )/\log
\left (2^{n}\right ).
\end{equation*}
The $L^{q}$-spectrum has gained some high attention from various authors
in recent years, e.g. \citep{MR3897401,MR3919361,KN2022}. Note that
$\beta _{\nu}$ is---as a limit superior of convex functions---itself a
convex function and that $\beta _{\nu}(0)$ is equal to the
\emph{upper Minkowski dimension} of $\supp \nu $ denoted by
$\overline{\dim}_{M}\left (\nu \right )$. Before stating our main result,
we introduce the following key quantity
\begin{equation*}
s_{b} :=  \inf \left \{ s>0:\beta _{\nu}(s)-bs\leq 0\right \} \;
\text{ for }b>0.
\end{equation*}

\begin{thm}
\label{thm:Estimation-n-Diameter}
Assuming \textup{\reftext{(\ref{eq:StandingAssumption})}}, we have
\begin{align*}
\overline{\mathbf{ord}}\left (\mathscr{S}W_{p}^{\ell},L_{\nu}^{q}
\right ) & \leq -\frac{1}{q\cdot s_{\varrho }}.
\end{align*}
\end{thm}

\begin{rem}
\label{rem1.2}
Note that
\begin{equation*}
-\frac{1}{q\cdot s_{\varrho }}\leq -
\frac{\varrho }{q\overline{\dim}_{M}\left (\nu \right )}-\frac{1}{q}
\leq \frac{1}{p}-\frac{1}{q}-\frac{\ell}{m}<0,
\end{equation*}
and that $-1/\left (q\cdot s_{\varrho }\right )=-\ell /m+1/p-1/q$ if and
only if $\beta _{\nu}(\text{s})=m\left (1-s\right )$ for some and hence
for all $s\in \left (0,1\right )$. These claims follow readily from the
convexity of $\beta _{\nu}$ by observing that for all
$s\in \left [0,1\right ]$ we have
$\beta _{\nu}(s)\leq \beta _{\nu}(0)\left (1-s\right )\leq m(1-s)$.

If $\nu $ has a non-trivial absolutely continuous part with respect to
Lebesgue, then an application of Jensen inequality guarantees
$\beta _{\nu}(s)=m\left (1-s\right )$, for all
$s\in \left (0,1\right )$. Hence, we gain a new perspective on the estimate
in \textup{\reftext{(\ref{eq:BirmanSolomanjakOld})}} in terms of the $L^{q}$-spectrum.
Namely, the intersection with the line through the origin with slope
$\varrho $ and $s\mapsto m(1-s)$ is given by $m/(m+\varrho )$, which leads
to the general upper bound $-\ell /m+1/p-1/q$ as obtained in
\citep{MR0217487} (for an illustration of this observation see \reftext{Fig.~\ref{fig:Moment-generating-function}} on page
\pageref{fig:Moment-generating-function}). Consequently, whenever
$\beta _{\nu}(s)<m\left (1-s\right )$, for some
$s\in \left (0,1\right )$, \reftext{Theorem~\ref{thm:Estimation-n-Diameter}} improves
the classical result of \citep{MR0217487} and
\citep[Theorem 5.1]{Borzov1971}. Indeed, a strict inequality occurs for
many singular measures, for example if
$\overline{\dim}_{M}\left (\nu \right )<m$. Roughly speaking, the more
nonuniform the mass of $\nu $ is distributed compared to the Lebesgue measure,
the faster
$\left (d_{n}\left (\mathscr{S}W_{p}^{\ell},L_{\nu}^{q}\right )
\right )_{n}$ decreases.%
\begin{figure}
\center{\begin{tikzpicture}[scale=1, every node/.style={transform shape},line cap=round,line join=round,>=triangle 45,x=1cm,y=1cm] \begin{axis}[ x=3.7cm,y=2.3cm, axis lines=middle, axis line style={very thick},ymajorgrids=false, xmajorgrids=false, grid style={thick,densely dotted,black!20}, xlabel= {$s$}, ylabel= {\;$\beta_\nu (s)$}, xmin=-0.4 , xmax=1.5 , ymin=-0.3, ymax=3.2,x tick style={color=black}, xtick={0, .425,0.6,1},xticklabels = {0,$s_\varrho$,$\frac{m}{m+\varrho}$,1},  ytick={0,1, 2,3},yticklabels = {0,1, $\overline{\dim}_M(\nu)$ ,3}] \clip(-0.5,-0.3) rectangle (4,4);
\draw[line width=1pt,smooth,samples=180,domain=-0.3:3.4] plot(\x,{log10(0.001^((\x))+0.28^((\x))+0.06^((\x))+0.659^((\x)))/log10(2)});
\draw [line width=01pt,dotted, domain=-0.05 :1.3] plot(\x,{3*(1-\x)});
\draw [line width=01pt,dashed, domain=-0.15 :1.3] plot(\x,{2*\x});

\node[circle,draw] (c) at (2.48 ,0 ){\,};

\draw [line width=.7pt,dotted, gray] (0.425 ,0.)--(0.425,1);
\draw [line width=.7pt,dotted, gray] (0.6  ,0 )-- (0.6,1.5);
\draw (1 ,1.91) node[anchor=north west] {{$s\mapsto\rho \cdot s$}};
\end{axis}
\end{tikzpicture}}
\caption{For $m=3$ the solid line illustrates the $L^{q}$-spectrum $\beta
_{\nu}$ for the self-similar measure $\nu $ supported on the\emph{
Sierpi\'{n}ski tetraeder} with all four contraction ratios equal $1/2$ and with
probability vector $\left (0.659,0.28,0.001,0.06\right )$; $\beta _{\nu}\left
(0\right )=\overline{\dim}_{M}\left (\nu \right )=2$. For $\varrho =2$ (slope of
the dashed line) the intersection of the spectrum and the dashed line determines
$s_{\varrho }$. The dotted line $s\protect\mapsto 3\left (1-s\right )$, which
coincides with the graph of $\beta _{\Lambda |_{\textbf{Q}}}$, intersects the
dashed line in $m/\left (m+\varrho \right )$ giving the upper
bound for $s_{\varrho }$.}
\label{fig:Moment-generating-function}
\end{figure}
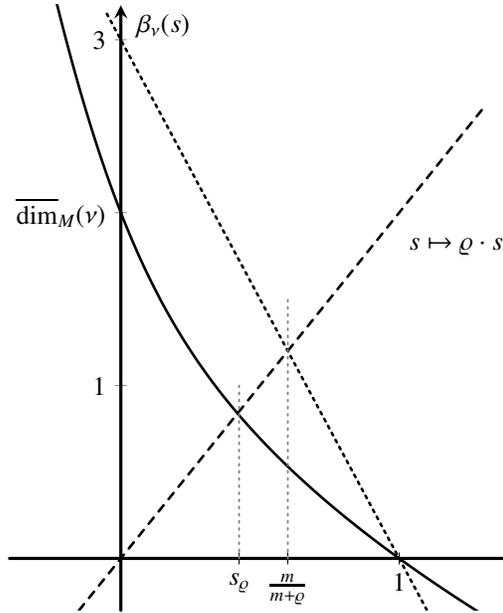%
\end{rem}

As a direct application of our result we consider polyharmonic operators,
i.e. we restrict to the Hilbert spaces setting $p=q=2$,
$H^{\ell} :=  W_{2}^{\ell}(\textbf{Q})$ and
$H_{0}^{\ell} :=  W_{0,2}^{\ell}(\textbf{Q})$: By \reftext{Theorem~\ref{thm:MainPolyharm_n-diameter}},
$d_{n-1}\left (\mathscr{S}H_{0}^{\ell},L_{\nu}^{2}\right )$ can be identified
with the square root of the $n$-th eigenvalue of the associated polyharmonic
operator with respect to $\nu $. This gives rise to improved upper bounds
of the decay rate of their eigenvalues (\reftext{Theorem~\ref{thm:Main_Polyharmonic}} in Section~\ref{subsec:General-setup-and_EVasymp}).
Further, using this connection for $\nu $ with non-trivial absolutely continuous
part with respect to Lebesgue, we deduce from
\citep[Theorem 5.1]{MR0278126} that $2s_{\varrho }=m/\ell $ and moreover,
for some explicit constant $c>0$,
\begin{equation*}
d_{n}\left (\mathscr{S}H_{0}^{\ell},L_{\nu}^{2}\right )\sim cn^{-
\ell /m}.
\end{equation*}
In Section~\ref{subsec:Application-to-self-similar} we consider self-similar
measures $\nu $ under the open set condition (cf. \reftext{Example~\ref{exa:IFS}}
for definitions); it follows from \citep{Nazarov} for $m=1$ and from \reftext{Theorem~\ref{thm:Nazarov_m>1}} for $m>1$ that
\begin{equation*}
\mathbf{ord}\left (\mathscr{S}H_{0}^{\ell},L_{\nu}^{2}\right )=-
\frac{1}{2s_{\varrho }}.
\end{equation*}
In Section~\ref{sec:Krein-Feller-operators-in} we finally consider polyharmonic
operators for the special case $\ell =m=1$. We will show that the associated
spectral problem is equivalent to the spectral problem of the classical
Kre\u{\i}n--Feller operator (see e.g.
\citep{KN21,MR2563669,MR2828537,MR3809018}). Using the superadditivity
established in \reftext{Theorem~\ref{thm:MainPolyharm_n-diameter}} (see also
\citep[Theorem 1.1]{KN2022}), we have equality in \reftext{Theorem~\ref{thm:Estimation-n-Diameter}} for \emph{all} finite Borel measure on
$\left (0,1\right )$, i.e.
\begin{equation*}
\overline{\mathbf{ord}}\left (\mathscr{S}H^{1},L_{\nu}^{2}\right )=
\overline{\mathbf{ord}}\left (\mathscr{S}H_{0}^{1},L_{\nu}^{2}\right )=-
\frac{1}{2s_{\varrho }}.
\end{equation*}

\section{Optimal partitions}
\label{sec:OptimalPartitions}

Fix $a>0$ and a Borel probability measure $\nu $ on $\textbf{Q}$. Let
$\Upsilon _{n}$ denote the set of all finite partitions consisting of at
most $n\in \mathbb N$ half-open $m$-dimensional subcubes of
$\textbf{Q}$. As in \citep{MR0217487,MR2864649}, we introduce an auxiliary
target quantity for the underlying optimisation problem given by the Kolmogorov
$n$-diameter: For $n\in \mathbb N,a>0$, and with
$\mathfrak{J}_{a}(Q) :=  \Lambda (Q)^{a}\nu (Q)$, $Q$ half-open subcube,
we let
%
\begin{align}
\gamma _{a,n} &  :=  \inf _{\Xi \in \Upsilon _{n}}\max _{Q\in
\Xi}\mathfrak{J}_{a}(Q),
\label{eq:optimierung}
\end{align}
and define the exponential growth rate of its reciprocal
\begin{equation*}
\alpha _{a} :=  \liminf _{n\rightarrow \infty}
\frac{\log \left (1/\gamma _{a,n}\right )}{\log (n)}.
\end{equation*}

\begin{rem}
\label{rem2.1}
The quantity $\gamma _{a,n}$ naturally arises in the study of approximation
order in $L_{\nu}^{2}$ of functions in $W_{p}^{\ell}(\textbf{Q})$ by piecewise
polynomial approximations (see for instance \citep{MR0278126} and \reftext{Proposition~\ref{prop:_PiecewiseApprox}})
as well as in the study of the spectral behaviour of polyharmonic operators
as defined in Section~\ref{subsec:Application-to-polyharmonic}, see also
\citep{MR0278126,MR0482138}. This common ground reveals a deep connection
between these two aspects. It is also worth pointing out that the so-called
quantization problem, that is the speed of approximation of a compactly
supported Borel probability measure by finitely supported measures (see
\citep{MR1764176} for an introduction), has also close links to the growth
rate of $n\gamma _{a,n}$. This will be subject of the forthcoming paper
\citep{KN22b}.

We make use of the fact that the asymptotic optimisation problem in
\textup{\reftext{(\ref{eq:optimierung})}} can, as a result of \reftext{Proposition~\ref{prop:_Elementary_Lem}},
be transformed into the following counting problem. Motivated by
\citep{KN2022,MR0217487,MR2083820}, we introduce follow quantities. Let
$\Pi $ denote the sets of all partitions of $\textbf{Q}$ by half-open
$m$-dimensional cubes. Then the exponential growth rate of
\begin{equation*}
\mathcal{N}_{a}\left (t\right ) :=  \inf \left \{ \card \left (P
\right ):P\in \Pi :\max _{Q\in P}\mathfrak{J}_{a}(Q)<1/t\right \} ,t>0,
\end{equation*}
given by
\begin{equation*}
h_{a} :=  \limsup _{t\to \infty}
\frac{\log \mathcal{N}_{a}\left (t\right )}{\log t},
\end{equation*}
will be called the\emph{ (upper) $\nu $-partition entropy with parameter
$a$. }Let us begin with the preparatory observation that for $a>0$, the
sequence $\left (\gamma _{a,2^{mn}}\right )$ is either strictly decreasing
or eventually constant zero.
\end{rem}

\begin{lem}
\label{lem:induction_betaN}%
For all $n\in \mathbb N$ and $a>0$, we have
$\gamma _{a,2^{m(n+1)}}\leq \frac{1}{2^{ma}}\gamma _{a,2^{mn}}$.
\end{lem}

\begin{proof}
For $\Xi \in \Upsilon _{2^{mn}}$ we can divide each $Q\in \Xi $ into
$2^{m}$ disjoint, equally sized, half-open cubes. The new resulting partition
denoted by $\Xi '$ satisfies
$\card (\Xi ')\leq 2^{m\left (n+1\right )}$ and
$\Lambda \left (Q\right )^{a}/2^{ma}=\Lambda \left (Q_{i}\right )^{a}$,
for all $Q'\subset Q\in \Xi $ with $Q'\in \Xi '$. This implies
\begin{equation*}
\max _{Q\in \Xi '}\mathfrak{J}_{a}(Q)\leq \frac{1}{2^{ma}}\max _{Q
\in \Xi}\mathfrak{J}_{a}(Q).\qedhere
\end{equation*}
\end{proof}
%
\begin{prop}
\label{prop:_Elementary_Lem}%
For $a>0$ we have $h_{a}=1/\alpha _{a}$.
\end{prop}

\begin{proof}
The proof follows along the same lines as the proof of the elementary Lemma
\citep[Lemma 2.2]{MR2083820}. First note that for $0<\varepsilon <1$ we
have
$\mathcal{N}_{a}\left (1/\varepsilon \right )=\inf \left \{ n\in
\mathbb N\mid \gamma _{a,n}<\varepsilon \right \} $. By \reftext{Lemma~\ref{lem:induction_betaN}} we have that
$\left (\gamma _{a,2^{mn}}\right )_{n}$ is a strictly decreasing null sequence
or eventually constant zero. The latter case is immediate. For the first
case the strict monotonicity gives as in \citep[Lemma 2.2]{MR2083820} for
$B\left (\varepsilon \right ) :=  \inf \left \{ n\in \mathbb N
\mid \gamma _{a,2^{mn}}<\varepsilon \right \} $
\begin{equation*}
\frac{1}{\alpha _{a}}=\limsup _{k\to \infty}
\frac{-\log k}{\log \gamma _{a,k}}=\limsup _{n\to \infty}
\frac{mn\log 2}{-\log \gamma _{a,2^{mn}}}=\limsup _{\varepsilon
\searrow 0}
\frac{B\left (\varepsilon \right )m\log 2}{-\log \varepsilon }=h_{a},
\end{equation*}
where the second equality follows by squeezing
$2^{m(n-1)}<k\leq 2^{mn}$, and the last equality by noting that
$2^{m\left (B\left (\varepsilon \right )-1\right )}\leq \mathcal{N}_{a}
\left (1/\varepsilon \right )\leq 2^{mB\left (\varepsilon \right )}$.
\end{proof}

\subsection{The $L^{q}$-spectrum and optimal partitions}
\label{sec2.1}

For $b>0$ and $n\in \mathbb N$, by the monotonicity and continuity of
$\beta _{n}^{\nu}$ there exists a unique number
$s_{n,b}\in \left [0,1\right ]$ such that
\begin{equation*}
\beta _{n}^{\nu}\left (s_{n,b}\right )=b\cdot s_{n,b}.
\end{equation*}

\begin{lem}
\label{lem2.4}
For all $b>0$,
\begin{equation*}
s_{b}=\limsup _{n\rightarrow \infty}s_{n,b}
\end{equation*}
and if $s_{b}>0$, then
$\beta _{\nu}\left (s_{b}\right )=b\cdot s_{b}$.
\end{lem}

\begin{rem}
\label{rem2.5}
The assumption $s_{b}>0$ cannot be dropped to guarantee the equality
$\beta _{\nu}\left (s_{b}\right )=b\cdot s_{b}$. In fact, for the finite
measure $\eta  :=  \sum p_{k}\delta _{x_{k}}$ with
$p_{k} :=  \mathrm{e}^{-k}$ and $x_{k} :=  1/k$ where
$\delta _{x}$ denotes the Dirac measure on $x$, we have
$\overline{\dim}_{M}(\eta )=1/2$,
\begin{equation*}
\beta _{\nu}(s)=
\begin{cases}
1/2 & ,s=0,
\\
0 & ,s>0
\end{cases}
\end{equation*}
and therefore $\beta _{\nu}\left (0\right )=1/2\neq b\cdot s_{b}=0$.
\end{rem}

\begin{proof}
Recall
$s_{b}=\inf \left \{ s>0:\beta _{\nu}\left (s\right )-b\cdot s\leq 0
\right \} $ and define
$s_{*} :=  \limsup _{n\rightarrow \infty}s_{n,b}$. Then for every
$\varepsilon >0$ and $n$ large enough we have
$s_{n,b}\leq s_{*}+\varepsilon $ and consequently
$\beta _{n}^{\nu}\left (s_{*}+\varepsilon \right )\leq \beta _{n}^{
\nu}\left (s_{n,b}\right )$. This implies
\begin{equation*}
\beta _{\nu}\left (s_{*}+\varepsilon \right )\leq b\cdot s_{*}.
\end{equation*}
If $s_{*}=0$, then $\beta _{\nu}\left (s\right )=0$ for all $s>0$, which
shows $s_{b}=0$. Assuming $s_{*}>0$, the continuity of
$\beta _{\nu}$ in $(0,1)$ gives
$\beta _{\nu}\left (s_{*}\right )\leq b\cdot s_{*}$. Let
$\left (n_{k}\right )_{k\in \mathbb N}$ be such that
$\lim _{k}s_{n_{k},b}=s_{*}$. Then for all
$\eta \in \left (0,s_{*}/2\right )$ and for $k$ large we have
$s_{n_{k},b}\geq s_{*}-\eta $, which gives
$\beta _{n_{k}}^{\nu}\left (s_{n_{k},b}\right )\leq \beta _{n_{k}}^{
\nu}\left (s_{*}-\eta \right )$. This implies
\begin{equation*}
bs_{*}\leq \limsup _{k}\beta _{n_{k}}^{\nu}\left (s_{*}-\eta \right )
\leq \beta _{\nu}\left (s_{*}-\eta \right ).
\end{equation*}
The continuity of $\beta _{\nu}$ in (0,1) gives
$bs_{*}=\beta _{\nu}\left (s_{*}\right )$ and therefore
$s_{b}=s_{*}$.
\end{proof}
The following lemma is the key to estimate $h_{a}$ in terms of the
$L^{q}$-spectrum $\beta _{\nu}$. To this end, for a given threshold
$t\in \left (0,\mathfrak{J}_{a}(\textbf{Q})\right )$, we will construct
partitions by dyadic cubes as a function of $t$ via an
\emph{adaptive approximation algorithm} in the sense of \citep{MR939183} (see
also \citep{MR1781213}) as follows. We say $Q\in \mathcal{D}$ is
\emph{bad, }if $\mathfrak{J}_{a}\left (Q\right )\geq t$, otherwise\emph{
}we call $Q$ \emph{good. }The goal is to construct a partition of
$\textbf{Q}$ with minimal cardinality, denoted by $P_{a,t}$, consisting
of elements of half-open dyadic cubes that are good. In the first step,
we divide $\textbf{Q}$ into $2^{d}$ half open cubes of equal size and move
good cubes among them to $P_{a,t}$. Now, repeat this procedure with respect
to each of the remaining bad cubes until no bad cubes are left. Since for
each $Q\in P_{a,t}$, we have
$\mathfrak{J}_{a}(Q)<t\leq \mathfrak{J}_{a}(Q')$, where $Q'$ denotes the
unique dyadic cube such that $Q'$ is the predecessor of $Q$. This ensures
that the procedure terminates after finitely many steps. The resulting
finite partition $P_{a,t}$ is optimal (in the sense of minimizing the cardinality)
among all partitions $P$ by half-open dyadic cubes fulfilling
$\max _{Q\in P}\mathfrak{J}_{a}(Q)<t$. This indicates that
$\card P_{a,t}$ provides a good approximation of
$\mathcal{N}_{a}\left (1/t\right )$. Now, the remaining task is to connect
the asymptotic behaviour of $\card (P_{a,t})$ with the $L^{q}$-spectrum
$\beta _{\nu}$. Motivated by ideas from large derivation theory and the
thermodynamic formalism \citep{MR2129258} we are able to bound
$h_{a}$ from above by $s_{am}$, namely, by comparing the cardinality of
$P_{a,t}$ and
$Q_{a,t} :=  \left \{ Q\in \mathcal{D}:\mathfrak{J}_{a}(Q)\geq t
\right \} $. This will be the key idea in the proof of \reftext{Lemma~\ref{lem:_EstimateGamma_n}}.
%
\begin{lem}
\label{lem:_EstimateGamma_n}%
For all $a>0$, we have
\begin{equation*}
h_{a}\leq s_{am}.
\end{equation*}
\end{lem}

\begin{proof}
Without loss of generality, we assume that $\nu $ is a probability measure.
For $t\in \left (0,1\right )$,
\begin{equation*}
P_{a,t}=\left \{ C\in \mathcal{D}:\mathfrak{J}_{a}\left (C\right )<t
\,\&\,\exists C'\in \mathcal{D}_{\left |\log _{2}\left (\Lambda (C)
\right )\right |/m-1}:C'\supset C\,\&\,\mathfrak{J}_{a}\left (C'
\right )\geq t\right \}
\end{equation*}
is a partition of $\textbf{Q}$ by dyadic cubes. With $Q_{a,t}$ as defined
above, we note that for $C\in P_{a,t}$ there is exactly one
$C'\in Q_{a,t}\cap \mathcal{D}_{\left |\log _{2}\left (\Lambda (C)
\right )\right |/m-1}$ with $C\subset C'$ and for each
$C'\in Q_{a,t}\cap \mathcal{D}_{\left |\log _{2}\left (\Lambda \left (C
\right )\right )\right |/m-1}$ there are at most $2^{m}$ elements of
$P_{a,t}\cap \mathcal{D}_{\left |\log _{2}\left (\Lambda (C)\right )
\right |/m}$ which are subsets of $C'$. Hence,
\begin{equation*}
\card P_{a,t}\leq 2^{m}\card Q_{a,t}.
\end{equation*}
By the definition of $s_{n,am}$, we have
\begin{equation*}
\sum _{C\in \mathcal{D}_{n}}\nu (C)^{s_{n,am}}=2^{am\cdot s_{n,am}}.
\end{equation*}
For $s>s_{am}$ and $\varepsilon  :=  \left (s-s_{am}\right )/2$, there
exists $K\in \mathbb N$ such that for all $k\geq K$ we have
$s_{am}+\varepsilon >s_{k,am}$. This gives
$s-\varepsilon =s_{am}+(s-s_{am})/2>s_{k,am}$ and we obtain for all
$0<t<1$,
\begin{align*}
t^{s}\card P_{a,t} & =\sum _{k=1}^{\infty}\sum _{C\in P_{a,t}\cap
\mathcal{D}_{k}}t^{s}\leq 2^{m}\sum _{k=1}^{K}\sum _{C\in Q_{a,t}
\cap \mathcal{D}_{k-1}}1
\\
&\quad {}+\sum _{k=K+1}^{\infty}2^{m}t^{s}\sum _{C
\in Q_{a,t}\cap \mathcal{D}_{k-1}}
\frac{\left (\mathfrak{J}_{a}\left (C\right )\right )^{s_{k-1,am}+\varepsilon}}{t^{s_{k-1,am}+\varepsilon}}
\\
& \leq 2^{m}\sum _{k=1}^{K}\sum _{C\in Q_{a,t}\cap \mathcal{D}_{k-1}}1
\\
&\quad {}+
\sum _{k=K+1}^{\infty}2^{m}t^{s}2^{-a(k-1)m\varepsilon}
\frac{2^{-am(k-1)s_{k-1,am}}2^{(k-1)am\cdot s_{k-1,am}}}{t^{s_{k-1,am}+\varepsilon}}
\\
& =2^{m}\sum _{k=1}^{K}\sum _{C\in Q_{a,t}\cap \mathcal{D}_{k-1}}1+
\sum _{k=K+1}^{\infty}2^{m(1-a(k-1)\varepsilon )}t^{s-\varepsilon -s_{k-1,am}}
\\
& \leq 2^{m}\sum _{k=0}^{K-1}\sum _{C\in \mathcal{D}_{k}}1+2^{m(1+a
\varepsilon )}\sum _{k=K+1}^{\infty}2^{-akm\varepsilon}<\infty .
\end{align*}
This implies
\begin{equation*}
\limsup _{t\searrow 0}
\frac{\log \left (\card P_{a,t}\right )}{-\log (t)}\leq s
\end{equation*}
and since $s>s_{am}$ was arbitrary,
\begin{equation*}
h_{a}=\limsup _{t\searrow 0}
\frac{\log \left (\mathcal{N}_{a}\left (1/t\right )\right )}{-\log (t)}
\leq \limsup _{t\searrow 0}
\frac{\log \left (\card P_{a,t}\right )}{-\log (t)}\leq s_{am}.\qedhere
\end{equation*}
\end{proof}
Now, we are in the position state one of our core results needed in the
proof of \reftext{Theorem~\ref{thm:Estimation-n-Diameter}}.
%
\begin{prop}
\label{prop:_CoreResult}%
For all $a>0$,
\begin{equation*}
\limsup _{n\rightarrow \infty}
\frac{\log \left (\gamma _{a,n}\right )}{\log (n)}=-\alpha _{a}=-
\frac{1}{h_{a}}\leq -\frac{1}{s_{am}}.
\end{equation*}
\end{prop}

\begin{proof}
This follows immediately from \reftext{Proposition~\ref{prop:_Elementary_Lem}} and
\reftext{Lemma~\ref{lem:_EstimateGamma_n}}.
\end{proof}
%
\begin{rem}
\label{rem2.8}
In the case $m=a=1$ we have shown in \citep{KN2022} that even equality
holds, i.e.
\begin{equation*}
\limsup _{n\rightarrow \infty}
\frac{\log \left (\gamma _{1,n}\right )}{\log (n)}=-\frac{1}{s_{1}}.
\end{equation*}
\end{rem}

\section{Approximation order}
\label{sec3}

\subsection{Piecewise polynomial approximations of functions of the Sobolev space
in the metric of $L_{\nu}^{q}$}
\label{subsec:Theorems-on-approximation}

In this section we recall some results of
\citep[\S 3]{MR0217487}
and \citep{MR0482138} which will be important for our applications to
$n$-diameters and polyharmonic operators. Let $Q\subset \textbf{Q}$ denote
a cube. As pointed out in the introduction our standing assumption
$\ell p/m>1$ ensures that $W_{p}^{\ell}\left (Q\right )$ is compactly embedded
in
$\left (\mathcal{C}(\overline{Q}),\left \Vert \,\cdot \,\right \Vert _{
\mathcal{C}(\overline{Q})}\right )$. In the case $\ell p/m\leq 1$ the situation
becomes more involved; in general, we have no compact embedding from
$W_{p}^{\ell}\left (Q\right )$ into $L_{\nu}^{2}(Q)$ (e.g.
\citep{MR2261337,MR1338787,MR817985}). Further, without loss of generality
we assume that $\nu $ is Borel probability measure on $\textbf{Q}$. For
every $u\in W_{p}^{\ell}\left (Q\right )$, we associate a polynomial
$r\in \mathbb{R}\left [x_{1},\ldots ,x_{m}\right ]$ of degree at most
$\ell -1$  satisfying the conditions
%
\begin{equation}
\int _{Q}x^{k}r(x)\;\mathrm d\Lambda (x)=\int _{Q}x^{k}u(x)\;
\mathrm d\Lambda (x)\,\,\text{for all}\,\,|k|\leq \ell -1.
\label{eq:Polynomial}
\end{equation}
By an application of\emph{ Hilbert's Projection Theorem} with respect to
$L^{2}(Q)$, we have that $r$ is uniquely determined by
\textup{\reftext{(\ref{eq:Polynomial})}} and set $P_{Q}u :=  r$. Note that
$P_{Q}$ defines a linear projection operator which maps from
$W_{p}^{\ell}\left (Q\right )$ to the finite-dimensional space of polynomials
in $m$ variables of degree not exceeding $\ell -1$ and we denote the dimension
of this finite dimensional space of polynomials by $\kappa $.

We finish this section with two crucial observations which follow from
\citep{MR0217487}.
%
\begin{lem}[{\citep[Lemma 3.1]{MR0217487}}]
\label{lem:_approxSupNorm}
Let $Q\subset \textbf{Q}$ be a cube. Then there exists $C_{1}>0$ independent
of $Q$ such that for all $u\in W_{p}^{\ell}\left (Q\right )$
\begin{equation*}
\left \Vert u-P_{Q}u\right \Vert _{\mathcal{C}(\overline{Q})}\leq C_{1}
\Lambda (Q)^{\ell /m-1/p}\left \Vert u\right \Vert _{L^{\ell ,p}(Q)}.
\end{equation*}
\end{lem}

\begin{defn}
\label{defn3.2}
Let $\Xi $ be a partition of $\textbf{Q}$ into half open cubes and we define
$\mathcal{P}\left (\Xi ,\ell -1\right )$ to be the space of piecewise-polynomial
functions which restrict on each cube $Q\in \Xi $ to a polynomial of degree
$\ell -1$. We define
\begin{align*}
P_{\Xi}:W_{p}^{\ell}\left (Q\right ) & \rightarrow \mathcal{P}\left (
\Xi ,\ell -1\right )
\\
u & \mapsto \sum _{Q\in \Xi}\mathbbm{1}_{Q}P_{Q}u,
\end{align*}
where $\mathbbm{1}_{Q}$ denotes the characteristic function on the cube
$Q$.
\end{defn}

\begin{prop}
\label{prop:_PiecewiseApprox}%
For a finite Borel measure $\nu $ on $\textbf{Q}$ and $1\le p\leq q$, there
exists $C_{2}>0$ such for all partitions $\Xi $ of $\textbf{Q}$ of half
open cubes and every $u\in W_{p}^{\ell}\left (\textbf{Q}\right )$, we have
\begin{equation*}
\left \Vert u-P_{\Xi}u\right \Vert _{L_{\nu}^{q}(\textbf{Q})}\leq C_{2}
\left \Vert u\right \Vert _{L^{\ell ,p}(\textbf{Q})}\left (\max _{Q
\in \Xi}\mathfrak{J}_{\varrho /m}(Q)\right )^{1/q}.
\end{equation*}
\end{prop}

\begin{proof}
This follows from the proof of \citep[Theorem 3.3]{MR0217487} using \reftext{Lemma~\ref{lem:_approxSupNorm}}.
\end{proof}

\subsection{Approximation order of Kolmogorov $n$-diameters}
\label{subsec:Approximation-order-of}

In this section we will prove our main result.
%
\begin{lem}
\label{lem:EstimateD_n}%
Under the assumption \textup{\reftext{(\ref{eq:StandingAssumption})}}, there exists
a constant $C_{3}>0$ depending only on $p,q,m,\ell $ such for all
$n\in \mathbb N$ we have
\begin{equation*}
d_{\kappa n}\left (\mathscr{S}W_{p}^{\ell},L_{\nu}^{q}\right )\leq C_{3}
\left (\gamma _{\varrho /m,\kappa n}\right )^{1/q}.
\end{equation*}
\end{lem}

\begin{proof}
For $n\in \mathbb N$ and $\Xi \in \Upsilon _{n}$ we have
$\text{dim}\mathcal{P}\left (\Xi ,\alpha \right )=\card (\Xi )\kappa
\leq n\kappa $ and therefore
\begin{equation*}
d_{\kappa n}\left (\mathscr{S}W_{p}^{\ell},L_{\nu}^{q}\right )\leq d_{
\kappa \card (\Xi )}\left (\mathscr{S}W_{p}^{\ell},L_{\nu}^{q}\right ).
\end{equation*}
Hence, we obtain by \reftext{Proposition~\ref{prop:_PiecewiseApprox}}
\begin{align*}
d_{\kappa \card (\Xi )}\left (\mathscr{S}W_{p}^{\ell},L_{\nu}^{q}
\right ) & \leq \sup _{u\in \mathscr{S}W_{p}^{\ell}}\inf _{y\in
\mathcal{P}\left (\Xi ,\ell -1\right )}\left \Vert u-y\right \Vert _{L_{
\nu}^{q}}\leq \sup _{u\in \mathscr{S}W_{p}^{\ell}}\left \Vert u-P_{
\Xi}u\right \Vert _{L_{\nu}^{q}}
\\
& \leq C_{2}\sup _{u\in \mathscr{S}W_{p}^{\ell}}\left \Vert u\right
\Vert _{L^{\ell ,p}(Q)}\left (\max _{Q\in \Xi}\mathfrak{J}_{q\ell /m-q/p}(Q)
\right )^{1/q}
\\
& \leq C_{2}\left (\max _{Q\in \Xi}\mathfrak{J}_{q\ell /m-q/p}(Q)
\right )^{1/q}.
\end{align*}
Taking the infimum over all partitions with cardinality less than or equal
to $n$ proves the lemma.
\end{proof}
We are now in the position to prove our main theorem.
\begin{proof}
[Proof of \reftext{Theorem~\ref{thm:Estimation-n-Diameter}}]
For $N\in \mathbb N$ and
$n(N) :=  \left \lfloor N/\kappa \right \rfloor $, \reftext{Lemma~\ref{lem:EstimateD_n}} gives
\begin{equation*}
d_{N}\left (\mathscr{S}W_{p}^{\ell},L_{\nu}^{q}\right )\leq d_{
\kappa n(N)}\left (\mathscr{S}W_{p}^{\ell},L_{\nu}^{q}\right )\leq C_{3}
\left (\gamma _{\varrho /m,\kappa n(N)}\right )^{1/q}.
\end{equation*}
Using $\kappa n(N)\leq N$, we obtain
\begin{equation*}
\frac{\log \left (d_{N}\left (\mathscr{S}W_{p}^{\ell},L_{\nu}^{q}\right )\right )}{\log (N)}
\leq
\frac{q \log (C_{3})+\log \left (\gamma _{\varrho /m,\kappa n(N)}\right )}{q\log (\kappa n(N))}
\end{equation*}
and with \reftext{Proposition~\ref{prop:_CoreResult}},
\begin{equation*}
\limsup _{N\rightarrow \infty}
\frac{\log \left (d_{N}\left (\mathscr{S}W_{p}^{\ell},L_{\nu}^{q}\right )\right )}{\log (N)}
\leq \limsup _{N\rightarrow \infty}
\frac{\log \left (\gamma _{\varrho /m,\kappa n(N)}\right )}{q\log (\kappa n(N))}
\leq -\frac{1}{q\cdot s_{\varrho }}.\qedhere
\end{equation*}
\end{proof}
In the following example we consider self-similar measure under the open
set condition. In this case the $L^{q}$-spectrum is well-known, allowing
us to give a formula of $s_{\varrho }$ only in terms of the probability
weights and contraction ratios.
%
\begin{example}
\label{exa:IFS}%
For fixed $n\in \mathbb N$ let $\left (T_{1},\ldots ,T_{n}\right )$ be
a set of contracting \emph{similarities} of $\mathbb{R}^{m}$ with ratios
$r_{1},\ldots ,r_{n}\in (0,1)$ that is for all
$x,y\in \mathbb{R}^{d}$ and $i=1,\ldots ,n$
\begin{equation*}
\left |T_{i}(x)-T_{i}(y)\right |=r_{i}\left |x-y\right |.
\end{equation*}
Furthermore, we assume the \emph{open set condition} (OSC) is fulfilled,
i.e. there exists an open set $O\subset \mathbb{R}^{m}$ such that
\begin{equation*}
T_{i}(O)\subset O\text{\:and\:}T_{i}\left (O\right )\cap T_{j}\left (O
\right )=\varnothing ,\:i\neq j.
\end{equation*}
Moreover, there exists a unique compact set $K$ such that
\begin{equation*}
K=\bigcup _{i=1}^{n}T_{i}(K).
\end{equation*}
Without loss of generality, we assume $K\subset (0,1)^{m}$. For
$(p_{1},\ldots ,p_{n})\in (0,1)^{n}$ let $\nu $ be the unique Borel measure
with
\begin{equation*}
\nu =\sum _{i=1}^{n}p_{i}\nu \circ T_{i}^{-1}.
\end{equation*}
The measure $\nu $ is called \emph{self-similar} measure with respect to
the weights $(p_{1},\ldots ,p_{n})$ and ratios
$(r_{1},\ldots ,r_{n})$ and we have $\supp \nu =K$. By
\citep[Theorem 16]{MR1312056} the $L^{q}$-spectrum $\beta _{\nu}$ on
$\mathbb{R}_{\geq 0}$ is given by the unique solution $s$ of
\begin{equation*}
\sum _{i=1}^{n}p_{i}^{s}r_{i}^{\beta _{\nu}(s)}=1.
\end{equation*}
Hence, under our standing assumptions ($p\leq q$ and $p\ell /m>1$) and
applying \reftext{Theorem~\ref{thm:Estimation-n-Diameter}}, we have
%
\begin{equation}
\overline{\mathbf{ord}}\left (\mathscr{S}W_{p}^{\ell},L_{\nu}^{q}
\right )\leq -\frac{1}{q\cdot s_{\varrho }},
\label{eq:LqUpperbound}
\end{equation}
where $s_{\varrho }$ is the unique solution $s$ of the equation
$\sum _{i=1}^{n}p_{i}^{s}r_{i}^{\varrho s}=1$. In particular, for the
\emph{`geometric'} choice of the weights
$p_{i} :=  r_{i}^{\delta}$, $i=1,\ldots ,n$, where
$\delta \in \left [0,m\right ]$ is Hausdorff dimension
$\dim _{H}\left (K\right )$ of $K$ determined as the unique solution of
$\sum _{i=1}^{n}r_{i}^{\delta}=1$, we obtain
\begin{equation*}
s_{\varrho }=\frac{\delta}{\varrho +\delta}.
\end{equation*}
Consequently, inserting $\varrho =\ell q-mq/p$ we get
\begin{equation*}
\overline{\mathbf{ord}}\left (\mathscr{S}W_{p}^{\ell},L_{\nu}^{q}
\right )\leq \frac{\ell}{\delta}\left (\frac{m}{\ell p}-1\right )-
\frac{1}{q}\leq \frac{\ell}{m}\left (\frac{m}{\ell p}-1\right )-
\frac{1}{q}=-\frac{\ell}{m}+\frac{1}{p}-\frac{1}{q}.
\end{equation*}
\end{example}

\section{Application to polyharmonic
operators}
\label{subsec:Application-to-polyharmonic}

\subsection{General setup and eigenvalue
asymptotics}
\label{subsec:General-setup-and_EVasymp}

In this section let $\nu $ be a finite Borel measure on
$\mathring{\textbf{Q}}$ and we restrict to the Hilbert space setting
$H_{0}^{\ell}$, respectively $H^{\ell}$. First, let us define the polyharmonic
operator as in \citep{MR0278126,Borzov1971,MR1298682,MR1328700}. We define the following quadratic forms
\begin{equation*}
J_{\nu}(u) :=  \int \left |u\right |^{2}\;\mathrm d\nu ,\, u\in L^2_\nu,\:\:\:I_{
\ell}(u) :=  \int _{\textbf{Q}}\sum _{|\alpha |=\ell}\left |D^{
\alpha}u\right |^{2}\;\mathrm d\Lambda ,\:u\in H_{0}^{\ell}
\end{equation*}
and let $J_{\nu}(u,v)$ and $I_{\ell}(u,v)$ denote the corresponding bilinear
forms. Observe that $I_{\ell}^{1/2}$ is an equivalent norm in
$H_{0}^{\ell}$ (see \citep[6.30 Theorem]{2003167}) and in virtue of our
standing assumption $2\ell /m>1$, we obtain that $H_{0}^{\ell}$ is compactly
embedded into
$\left (\mathcal{C}\left (\overline{\textbf{Q}}\right ),\left \Vert
\,\cdot \,\right \Vert _{\mathcal{C}\left (\overline{\textbf{Q}}
\right )}\right )$ (see e.g.
\citep[Theorem 6.3, Part II]{2003167}). In particular, there is a constant
$C>0$ such that for all $u\in H_{0}^{\ell}$
\begin{equation*}
J_{\nu}(u)\leq CI_{\ell}(u)
\end{equation*}
and by an application of the Cauchy-Schwarz inequality, for fixed
$u\in H_{0}^{\ell}$, the map $w\mapsto J_{\nu}(u,w)$ defines a bounded
linear functional. By the Riesz Representation Theorem, we can define a
bounded linear non-negative self-adjoint operator $T_{\nu}$ mapping from
$\left (H_{0}^{\ell},I_{\ell}\left (\cdot ,\cdot \right )\right )$ to itself
such that, for all $u,w\in H_{0}^{\ell}$,
\begin{equation*}
J_{\nu}(u,w)=I_{\ell}\left (T_{\nu}(u),w\right )
\end{equation*}
and
\begin{equation*}
\sqrt{I_{\ell}\left (T_{\nu}(u),T_{\nu}(u)\right )}=\left \Vert J_{
\nu}\left (u,\,\cdot \,\right )\right \Vert  :=  \sup _{y\in H_{0}^{
\ell}\setminus \left \{ 0\right \} }
\frac{\left |J_{\nu}(u,y)\right |}{I_{\ell}(y,y)^{1/2}}\leq C
\sqrt{I_{\ell}(u,u)}.
\end{equation*}
To finally show that $T_{\nu}$ is compact, let
$\left (u_{n}\right )_{n\in \mathbb N}$ be a bounded sequence in
$\left (H_{0}^{\ell},I_{\ell}\right )$. Then, by the compact embedding
of $H_{0}^{\ell}$ into
$\left (\mathcal{C}\left (\overline{\textbf{Q}}\right ),\left \Vert
\,\cdot \,\right \Vert _{\mathcal{C}\left (\overline{\textbf{Q}}
\right )}\right )$, there exists a subsequence
$\left (u_{n_{k}}\right )_{k\in \mathbb N}$, which is a Cauchy in
$\left (\mathcal{C}\left (\overline{\textbf{Q}}\right ),\left \Vert
\,\cdot \,\right \Vert _{\mathcal{C}\left (\overline{\textbf{Q}}
\right )}\right )$. Hence, for all $k,m\in \mathbb N$, we have
\begin{align*}
&I_{\ell}\left (T_{\nu}(u_{n_{k}})-T_{\nu}(u_{n_{m}})\right )
\\
 &\quad  =\int _{
\textbf{Q}}\left (u_{n_{k}}-u_{n_{m}}\right )T_{\nu}\left (u_{n_{k}}-u_{n_{m}}
\right )\;\mathrm d\nu
\\
&\quad  \leq \left \Vert u_{k}-u_{n_{m}}\right \Vert _{\mathcal{C}\left (
\overline{\textbf{Q}}\right )}\sqrt{\nu (\textbf{Q})}
\frac{\left (\int _{Q}T_{\nu}\left (u_{n_{k}}-u_{n_{m}}\right )^{2}\;\mathrm d\nu \right )^{1/2}}{\sqrt{I_{\ell}\left (u_{n_{k}}-u_{n_{m}}\right )}}
\sqrt{I_{\ell}\left (u_{n_{k}}-u_{n_{m}}\right )}
\\
&\quad  \leq \left \Vert u_{k}-u_{n_{m}}\right \Vert _{\mathcal{C}\left (
\overline{\textbf{Q}}\right )}\sqrt{C\nu (\textbf{Q})}
\frac{\sqrt{I_{\ell}\left (T_{\nu}\left (u_{n_{k}}-u_{n_{m}}\right )\right )}}{\sqrt{I_{\ell}\left (u_{n_{k}}-u_{n_{m}}\right )}}
\sqrt{I_{\ell}\left (u_{n_{k}}-u_{n_{m}}\right )}
\\
&\quad  \leq \left \Vert u_{k}-u_{n_{m}}\right \Vert _{\mathcal{C}\left (
\overline{\textbf{Q}}\right )}
\sqrt{C^{3}\nu (\textbf{Q})I_{\ell}\left (u_{n_{k}}-u_{n_{m}}\right )},
\end{align*}
taking into account that the sequence
$\left (u_{n}\right )_{n\in \mathbb N}$ is bounded with respect to
$I_{\ell}$, we deduce that
$\left (T_{\nu}(u_{n_{k}})\right )_{k\text{$\in \mathbb N$}}$ is a Cauchy
sequence in the Hilbert space $\left (H_{0}^{\ell},I_{\ell}\right )$.
%
\begin{defn}
\label{defn4.1}
An element $f\in H_{0}^{\ell}\setminus \left \{ 0\right \} $ is called
\emph{eigenfunction }of $T_{\nu}$ with \emph{eigenvalue} $\lambda $, if for
all $g\in H_{0}^{\ell}$, we have
\begin{equation*}
J_{\nu}(f,g)=\lambda I_{\ell}(f,g).
\end{equation*}
\end{defn}

From the spectral theorem for self-adjoint compact operator we deduce that
there is a decreasing sequence of non-negative eigenvalues
$\left (\lambda _{n}^{\nu}\right )_{n\in \mathbb N}$ tending to $0$. We
are interested in the decay rate of the sequence
$\left (\lambda _{n}^{\nu}\right )_{n\in \mathbb N}$. Note that
$\ker \left (T_{\nu}\right )$ can be quite large; for example in the case
that $\nu $ equals the Dirac measure $\delta _{1/2}$ on $(0,1)$, we have
$\ker \left (T_{\nu}\right )=\left \{ f\in H_{0}^{1}(0,1)\mid f(1/2)=0
\right \} $ and there is exactly one eigenvalue not equal to zero, namely
$\lambda =1/4$ with (normalized) eigenfunction
$f_{1/4}(x) :=  \mathbbm{1}_{(0,1/2)}2x+\mathbbm{1}_{[1/2,1)}
\left (1-2x\right )$.

As we will see, the growth rate of the eigenvalues is encoded by the $L^{q}$-spectrum.  
Using the variational principle it can be shown that the eigenvalues can
be computed in terms of $n$-diameter (e.g.
\citep[Theorem 4.5]{MR2517942}).
%
\begin{thm}
\label{thm:MainPolyharm_n-diameter}%
Let $\lambda _{n}^{\nu}\searrow 0$ be the decreasing sequence of eigenvalues
of the polyharmonic operator $T_{\nu}$ with respect to the Borel measure
$\nu $. Then we have
\begin{equation*}
\sqrt{\lambda _{n+1}^{\nu}}=d_{n}\left (\mathscr{S}H_{0}^{\ell},L_{
\nu}^{2}\right ).
\end{equation*}
\end{thm}

Combining \reftext{Theorem~\ref{thm:MainPolyharm_n-diameter}}, \reftext{Theorem~\ref{thm:Estimation-n-Diameter}} and the fact
$d_{n}\left (\mathscr{S}H_{0}^{\ell},L_{\nu}^{2}\right )\leq d_{n}
\left (\mathscr{S}H^{\ell},L_{\nu}^{2}\right )$, we obtain the following
upper bound.
%
\begin{thm}
\label{thm:Main_Polyharmonic}%
We have
\begin{equation*}
\limsup _{n\rightarrow \infty}
\frac{\log \left (\lambda _{n}^{\nu}\right )}{\log (n)}\leq -
\frac{1}{s_{2\ell -m}}\leq -\left (
\frac{2\ell -m}{\overline{\dim}_{M}\left (\nu \right )}+1\right )
\leq -\frac{2\ell}{m}.
\end{equation*}
\end{thm}

\begin{rem}
\label{rem4.4}
If $\nu $ is a singular measure with respect to the Lebesgue measure, the
result of \citep{Borzov1971} gives us
\begin{equation*}
\lambda _{n}^{\nu}=o\left (n^{^{-2\ell /m}}\right ).
\end{equation*}
In the case $\beta _{\nu}\left (s\right )<m(1-s)$, for some
$s\in (0,1)$, we can improve this estimate, in fact we have for every
$\varepsilon >0$
\begin{equation*}
\lambda _{n}^{\nu}=O\left (n^{-1/s_{2\ell -m}+\varepsilon}\right ).
\end{equation*}
In general one cannot expect that $\lambda _{n}^{\nu}\asymp n^{-s}$ for
some $s>0$ (see for example \citep{Arzt_diss} or \citep{KN2022}).
\end{rem}

\subsection{Application to self-similar measures}
\label{subsec:Application-to-self-similar}

To treat self-similar measures $\nu $ in more detail, we need the following
characterisation of the eigenvalues of the linear self-adjoint compact
operator $T_{\nu}$ by the well-known max-min principle (see for example
\citep[Section 4]{MR831201}, \citep[Theorem 2.1, page 64]{MR774404}).
%
\begin{prop}
\label{prop:max-min}
For all $i\in \mathbb N$, we have
\begin{equation*}
\lambda _{i}^{\nu}=\sup \left \{ \inf _{\psi \in G\setminus \left \{ 0
\right \} }\left \{ J_{\nu}\left (\psi ,\psi \right ):\,I_{\ell}(
\psi ,\psi )=1\right \} \colon G\subset H_{0}^{1},\:i
\text{-dimensional\,subspace}\right \} .
\end{equation*}
\end{prop}

\begin{proof}
Let $(e_{j})_{j\in \mathbb N}$ be an orthonormal basis of eigenfunctions
of $T_{\nu}$ corresponding to the eigenvalues
$\left (\lambda _{j}^{\nu}\right ){}_{j\in \mathbb N}$. Let $G_{i}$ be
an $i$-dimensional subspace of $H_{0}^{1}$ and define
$E_{i} :=  \overline{\spann \left (e_{j}:j\geq i\right )}$. Observe
that $\dim \left (H_{0}^{1}/E_{i}\right )=i-1$. Using the first and second
isomorphism theorem,
\begin{equation*}
i-\dim (G_{i}\cap E_{i})=\dim (G_{i}/G_{i}\cap E_{i})=\dim \left (
\left (G_{i}+E_{i}\right )/G_{i}\right )\leq \dim \left (H_{0}^{1}/G_{i}
\right )=i-1.
\end{equation*}
Therefore, there exists
$u\in E_{i}\cap G_{i}\neq \left \{ 0\right \} $ with
$I_{\ell}(u,u)=1$ and we may write $u=\sum _{k\geq i}c_{k}e_{k}$ with
$\sum _{k\geq i}c_{k}^{2}=1$. Consequently,
\begin{align*}
\inf _{\psi \in G_{i}\setminus \left \{ 0\right \} }\left \{ J_{\nu}
\left (\psi ,\psi \right ):\,I_{\ell}(\psi ,\psi )=1\right \} & \leq J_{
\nu}\left (u,u\right )=I_{\ell}\left (T_{\nu}(u),u\right )
\\
& =I_{\ell}\left (\sum _{k\geq i}\lambda _{k}^{\nu}c_{k}e_{k},\sum _{k
\geq i}c_{k}e_{k}\right )\leq \lambda _{i}^{\nu}.
\end{align*}
Conversely, for
$G_{i} :=  \spann \left \{ e_{1},\ldots ,e_{i}\right \} $, we have
$\inf _{\psi \in G_{i}\setminus \left \{ 0\right \} }\left \{ J_{\nu}
\left (\psi ,\psi \right ):\,I_{\ell}(\psi ,\psi )=1\right \} =
\lambda _{i}^{\nu}$.
\end{proof}
The following theorem has been established in
\citep[Theorem 3.1]{Nazarov} for dimension $m=1$. We prove the corresponding
result for arbitrary dimension $m\in \mathbb N$ of the ambient space.
%
\begin{thm}
\label{thm:Nazarov_m>1}%
Let $\nu $ be a self-similar measure under OSC with feasible open set
$O\subset (0,1)^{m}$ as defined in \reftext{Example~\ref{exa:IFS}},
$\ell \in \mathbb N$, and assume $\nu \left (\partial O\right )=0$. Then,
\begin{equation*}
\mathbf{ord}\left (\mathscr{S}H_{0}^{\ell},L_{\nu}^{2}\right )=
\mathbf{ord}\left (\mathscr{S}H^{\ell},L_{\nu}^{2}\right )=\lim _{l
\to \infty}\frac{\log \left (\lambda _{l}^{\nu}\right )}{\log (l)}=-
\frac{1}{2s_{\varrho }}.
\end{equation*}
\end{thm}

\begin{proof}
Here, we follow partially \citep{MR1338787}. For $t$ with
$0<t<\min p_{i}^{s_{\varrho }}r_{i}^{(2\ell -m)s_{\varrho }}$ let
\begin{equation*}
E_{t} :=  \left \{ \omega \in I^{*}\mid p_{\omega}r_{\omega}^{2
\ell -m}<t\leq p_{\omega ^{-}}r_{\omega ^{-}}^{2\ell -m}\right \} ,
\end{equation*}
where
$I^{*}=\bigcup _{l\in \mathbb N}\left \{ 1,\ldots ,n\right \} ^{l}$ and
$p_{\omega}r_{\omega}^{2\ell -m} :=  \prod _{i=1}^{k}p_{\omega _{i}}r_{
\omega _{i}}^{2\ell -m}$,
$\omega =\omega _{1}\cdot \ldots \cdot \omega _{k}\in I^{*}$,
$k\in \mathbb N$. Then, by our assumption
$\nu \left (\partial O\right )=0$, for all $\omega \in E_{t}$, it follows
that $\nu \left (\partial T_{\omega}(O)\right )=0$ and
\begin{equation*}
\nu \left (\bigcup _{\omega \in E_{t}}T_{\omega}(O)\right )=1
\text{\:and\:}T_{\omega}(O)\cap T_{\omega '}(O)=\varnothing
\text{\:for all\:}\omega ,\omega '\in E_{t}\:\text{with}\:\omega \neq
\omega '.
\end{equation*}
Hence,
\begin{equation*}
\sum _{\omega \in E_{t}}p_{\omega}^{s_{\varrho }}r_{\omega}^{(2\ell -m)s_{
\varrho }}=1\leq t^{s_{\varrho }}\card \left (E_{t}\right ).
\end{equation*}
Fix $a\in K$ and choose $u_{0}\in \mathcal{C}_{c}^{\infty}(O)$ such that
$u_{0}(a)>0$. For $\omega \in E_{t}$, we set
\begin{equation*}
u_{\omega}(x) :=
\begin{cases}
u_{0}\left (T_{\omega}^{-1}(x)\right ) & ,x\in T_{\omega}(O)
\\
0 & ,x\in [0,1]^{d}\setminus T_{\omega}(O)
\end{cases}
.
\end{equation*}
Then
$u_{\omega}\in \mathcal{C}_{c}^{\infty}\left (T_{\omega}(O)\right )$. Since
the supports of $(u_{\omega})_{\omega \in E_{t}}$ are disjoint, it follows
that the $(u_{\omega})_{\omega \in E_{t}}$ are mutually orthogonal both
in $L_{\nu}^{2}$ and in $H_{0}^{\ell}$, and
$\spann (u_{\omega}:\omega \in E_{t})$ is therefore a
$\card (E_{t})$-dimensional subspace of $H_{0}^{\ell}$. Moreover, we have
\begin{equation*}
J_{\nu}\left (u_{\omega}\right )=p_{\omega}\int u_{0}^{2}\;\mathrm d
\nu \:\text{and}\:I_{\ell}\left (u_{\omega}\right )=r_{\omega}^{d-2
\ell}I_{\ell}\left (u_{0}\right ).
\end{equation*}
We obtain
\begin{equation*}
\frac{J_{\nu}\left (u_{\omega}\right )}{I_{\ell}\left (u_{\omega}\right )}=r_{
\omega}^{2\ell -m}p_{\omega}
\underbrace{\frac{J_{\nu}\left (u_{0}\right )}{I_{\ell}\left (u_{0}\right )}}_{
 =:  R}.
\end{equation*}
Now, for
$u=\sum _{\omega \in E_{t}}c_{\omega}u_{\omega}\in H_{0}^{\ell}
\setminus \left \{ 0\right \} $ with $c_{\omega}\in \mathbb{R}$. we have
\begin{equation*}
\frac{J_{\nu}\left (u_{\omega}\right )}{I_{\ell}\left (u_{\text{$\omega$}}\right )}=
\frac{\sum _{\omega \in E_{t}}c_{\omega}^{2}J_{\nu}\left (u_{\omega}\right )}{\sum _{\omega \in E_{t}}c_{\omega}^{2}I_{\ell}\left (u_{\omega}\right )}=R
\frac{\sum _{\omega \in E_{t}}c_{\omega}^{2}p_{\omega}r_{\omega}^{2\ell -m}I_{\ell}\left (u_{\omega}\right )}{\sum _{\omega \in E_{t}}c_{\omega}^{2}I_{\ell}\left (u_{\omega}\right )}
\geq tR\min p_{i}r_{i}^{2\ell -m}.
\end{equation*}
The min-max principle stated in \reftext{Proposition~\ref{prop:max-min}} gives
\begin{equation*}
tR\min p_{i}r_{i}^{2\ell -m}\leq \lambda _{\card (E_{t})}^{\nu}\leq
\lambda _{\left \lfloor t^{-s_{\varrho }}\right \rfloor }^{\nu}.
\end{equation*}
In particular, for $t=l^{-1/s_{\varrho }}$ and $l\in \mathbb N$ large,
\begin{equation*}
l^{-1/s_{\varrho }}R\min p_{i}r_{i}^{2\ell -m}\leq \lambda _{l}^{\nu},
\end{equation*}
Combining this with \textup{\reftext{(\ref{eq:LqUpperbound})}}, gives
\begin{equation*}
-\frac{1}{s_{\varrho }}\leq \liminf _{l\rightarrow \infty}
\frac{\log \left (\lambda _{l}^{\nu}\right )}{\log (l)}\leq \limsup _{m
\rightarrow \infty}
\frac{\log \left (\lambda _{l}^{\nu}\right )}{\log (l)}\leq -
\frac{1}{s_{\varrho }}.\qedhere
\end{equation*}
\end{proof}

\subsection{One-dimensional Kre\u{\i}n--Feller operators}
\label{sec:Krein-Feller-operators-in}

In this final section, we show that the spectral problem of the Kre\u{\i}n--Feller
operator in dimension one is actually equivalent to the spectral problem
of polyharmonic operator $T_{\nu}$ (excluding the eigenvalue zero) for
the case $\ell =m=1$ and $p=q=2$. Let us recall the general setting for
the one-dimensional Kre\u{\i}n--Feller Operator with respect to the finite
Borel measure $\nu $ on $\left (0,1\right )$. We set
\begin{equation*}
\mathcal{C}_{\nu}([0,1]) :=  \left \{ f\in \mathcal{C}([0,1])
\mid f\:\text{\text{is}\:affine\:linear\:on\:the\:components\:of\:}[0,1]
\setminus \supp (\nu )\right \}
\end{equation*}
and
$\dom (\mathcal{E}_{\nu}) :=  H_{0}^{1}\cap \mathcal{C}_{\nu}([0,1])$
with Dirichlet form
$\mathcal{E}_{\nu}\left (f,g\right ) :=  \int _{(0,1)}\nabla f
\nabla u\;\mathrm d\Lambda $. Now, we consider the spectral problem of
the classical Kre\u{\i}n--Feller operator considered in
\citep{KN21,KN2022,MR2828537}. We call
$u\in \dom (\mathcal{E}_{\nu})\setminus \left \{ 0\right \} $ an\emph{ eigenfunction
}with \emph{eigenvalue} $\lambda $ if
%
\begin{equation}
\int _{(0,1)}\nabla f\nabla u\;\mathrm d\Lambda =\lambda \int fu\;
\mathrm d\nu ,
\label{eq:Eigenfunction}
\end{equation}
for all $f\in \dom (\mathcal{E}_{\nu})$. We need a decomposition
\begin{equation*}
[0,1]\setminus \supp (\nu ) =:  A_{1}\cup A_{2}\cup \bigcup _{i
\in I}(a_{i},b_{i}),
\end{equation*}
where $I\subset \mathbb N$, $A_{1} :=  [0,d_{1})$ if
$0\notin \supp (\nu )$ otherwise $A_{1}=\varnothing $,
$A_{2} :=  (d_{2},1]$ if $1\notin \supp (\nu )$ otherwise
$A_{2}=\varnothing $, and the intervals $[0,c_{1})$, $(c_{2},1]$,
$(a_{i},b_{i})$, $i\in I$, are mutually disjoint.

The following Lemma will provide a map from $H_{0}^{1}$ to
$\dom \left (\mathcal{E}_{\nu}\right )$.
%
\begin{lem}[{\citep[Lemma 2.1]{KN2022}}]
\label{lem:RepraesentantL2NU}%
The map
$\tau _{\nu}:H_{0}^{1}\rightarrow \dom \left (\mathcal{E}_{\nu}
\right )$
\begin{equation*}
\tau _{\nu}(f)(x) :=
\begin{cases}
f(a_{i})+\frac{f(b_{i})-f(a_{i})}{b_{i}-a_{i}}\left (x-a_{i}\right ), &
x\in (a_{i},b_{i}),\:i\in N,
\\
f(x), & x\in \supp (\nu ),
\\
\frac{f(c_{1})}{c_{1}}(x), & x\in [0,c_{1}),\:0\notin \supp (\nu ),
\\
\frac{f(c_{2})}{1-c_{2}}\left (1-x\right ), & x\in (c_{2},1],\:1
\notin \supp (\nu ),
\end{cases}
\end{equation*}
is surjective, $\tau _{\nu}(f)=f$ as elements of $L_{\nu}^{2}$, and we
have
\begin{equation*}
\nabla \tau _{\nu}(f)(x) :=
\begin{cases}
\frac{f(b_{i})-f(a_{i})}{b_{i}-a_{i}}, & x\in (a_{i},b_{i}),\ i\in N,
\\
\nabla f(x), & x\in \supp (\nu ),
\\
\frac{f(c_{1})}{c_{1}}, & x\in [0,c_{1}),\ 0\notin \supp (\nu ),
\\
-\frac{f(c_{2})}{1-c_{2}}, & x\in (c_{2},1],\ 1\notin \supp (\nu ).
\end{cases}
\end{equation*}
\end{lem}

\begin{lem}
\label{lem:AequivalenzSolo}%
We have $\varphi \in \dom (\mathcal{E}_{\nu})$ is an eigenfunction
of \textup{\reftext{(\ref{eq:Eigenfunction})}} with eigenvalue $\lambda $, if and
only if, for all $g\in H_{0}^{1}$, we have
\begin{equation*}
\int _{[0,1]}\nabla \varphi \nabla g\;\mathrm d\Lambda =\lambda
\cdot \int \varphi g\;\mathrm d\nu .
\end{equation*}
\end{lem}

\begin{proof}
Recall,
$[0,1]\setminus \supp (\nu )=A_{1}\cup A_{2}\cup \bigcup _{i\in I}(a_{i},b_{i})$
and define
$c_{i}=\nabla \varphi \left (\frac{a_{i}+b_{i}}{2}\right )$. For simplicity
we assume $0,1\in \supp (\nu )$ and let $g\in H_{0}^{1}$ be, then we have
\begin{align*}
\int _{[0,1]}\nabla \varphi \nabla g\;\mathrm d\text{$\Lambda$} & =
\int _{\supp (\nu )}\nabla \varphi \nabla g\;\mathrm d
\text{$\Lambda$}+\sum _{i\in I}\int _{(a_{i},b_{i})}c_{i}\nabla g\;
\mathrm d\Lambda
\\
& =\int _{\supp (\nu )}\nabla \varphi \nabla g\;\mathrm d
\text{$\Lambda$}+\sum _{i\in I}c_{i}(g(b_{i})-g(a_{i}))
\\
& =\int _{\supp (\nu )}\nabla \varphi \nabla g\;\mathrm d
\text{$\Lambda$}+\sum _{i\in I}c_{i}(b_{i}-a_{i})\left (
\frac{g(b_{i})-g(a_{i})}{b_{i}-a_{i}}\right )
\\
& =\int _{\supp (\nu )}\nabla \varphi \nabla \tau (g)\;\mathrm d
\text{$\Lambda$}+\sum _{i\in I}\int _{(a_{i},b_{i})}\nabla \varphi
\nabla \tau (g)\;\mathrm d\Lambda
\\
& =\int _{(0,1)}\nabla \varphi \nabla \tau (g)\;\mathrm d
\text{$\Lambda$}=\lambda \int fg\;\mathrm d\text{$\nu$.}\qedhere
\end{align*}
\end{proof}
%
\begin{prop}
\label{Prop:SolomanyakOperator}%
Let $\varphi $ be an eigenfunction of $T_{\nu}$ with eigenvalue
$\lambda >0$. Then $\varphi $ is affine linear on the connected components
of $[0,1]\setminus \supp (\nu )$.
\end{prop}

\begin{proof}
Here we closely follow \citep[Proposition 3.2]{MR4241300}. Let
$(a,b)$ be a component $[0,1]\setminus \supp (\nu )$. Then we have for
all $f\in H_{0}^{1}$
\begin{equation*}
\lambda \int _{[0,1]}\nabla f\nabla \varphi \;\mathrm d\Lambda =\int f
\varphi \;\mathrm d\nu .
\end{equation*}
For $x_{1}x_{2}\in (a,b)$, $x_{1}<x_{2}$ and for all $\delta >0$ sufficiently
small such that
\begin{equation*}
a<x_{1}-\delta <x_{1}<x_{1}+\delta <x_{2}-\delta <x_{2}<x_{2}+\delta <b,
\end{equation*}
we have that
\begin{equation*}
g:x\mapsto \frac{x-\left (x_{1}-\delta \right )}{2\delta}\mathbbm{1}_{(x_{1}-
\delta ,x_{1}+\delta )}(x)+\mathbbm{1}_{\left [x_{1}+\delta ,x_{2}-
\delta \right ]}(x)+\frac{x_{2}+\delta -x}{2\delta}\mathbbm{1}_{(x_{2}-
\delta ,x_{2}+\delta )}(x).
\end{equation*}
defines an element in $H_{0}^{1}$. Hence, we obtain
\begin{align*}
0 & =\int _{[0,1]}\nabla g\nabla \varphi \;\mathrm d\Lambda =
\frac{1}{2\delta}\int _{(x_{1}-\delta ,x_{1}+\delta )}\nabla \varphi
\;\mathrm d\Lambda -\frac{1}{2\delta}\int _{(x_{2}-\delta ,x_{2}+
\delta )}\nabla \varphi \;\mathrm d\Lambda .
\end{align*}
The Lebesgue differentiation theorem forces $\nabla \varphi (x)=c$ almost
everywhere. For all $x\in (a,b)$, we obtain
\begin{equation*}
\varphi (x)=\varphi (a)+\int _{[a,x]}\nabla \varphi \;\mathrm d
\Lambda =\varphi (a)+c(x-a).
\end{equation*}
The following proposition shows that the equivalence of the spectral problems.
\end{proof}
%
\begin{prop}
\label{prop:_ProblemeAequivlaent}
We have that $\lambda >0$ is an eigenvalue of $T_{\nu}$ if and only if
$1/\lambda $ is an eigenvalue of \textup{\reftext{(\ref{eq:Eigenfunction})}}.
\end{prop}

\begin{proof}
Let $\varphi $ be an eigenfunction of $T_{\nu}$ with eigenvalue
$\lambda >0$. Using \ref{Prop:SolomanyakOperator} it follows
$\varphi \in \dom \left (\mathcal{E}_{\nu}\right )$ and by definition we
have for all
$f\in \dom \left (\mathcal{E}_{\nu}\right )\subset H_{0}^{1}$
\begin{equation*}
\lambda \int _{[0,1]}\nabla f\nabla \varphi \;\mathrm d\Lambda =\int f
\varphi \;\mathrm d\nu .
\end{equation*}
Hence, $\varphi $ is an eigenfunction of
\textup{\reftext{(\ref{eq:Eigenfunction})}} with eigenvalue $1/\lambda $.

Reversely, let $\varphi \in \dom (\mathcal{E}_{\nu})$ be an eigenfunction
of \textup{\reftext{(\ref{eq:Eigenfunction})}} with eigenvalue $\lambda $. Then it
follows $\lambda >0$ and by \reftext{Lemma~\ref{lem:AequivalenzSolo}} we have for
all $f\in H_{0}^{1}$
\begin{equation*}
\int _{[0,1]}\nabla f\nabla \varphi \;\mathrm d\Lambda =\lambda \int f
\varphi \;\mathrm d\nu ,
\end{equation*}
which shows that $\varphi $ is an eigenfunction of $T_{\nu}$ with eigenvalue
$1/\lambda $.
\end{proof}
We end this section by using the above observation to give a short proof
of the sub-/superadditivity of the eigenvalue counting function announced
in the introduction.

Now, for $d_{0}=0<d_{1}<\dots <d_{n}<d_{n+1}=1$ with
$\nu \left (\left \{ d_{k}\right \} \right )=0$, we define the following
closed subspace of $H_{0}^{1}$ given by
\begin{equation*}
F :=  \left \{ u\in H_{0}^{1}:u(d_{i})=0,i\in \left \{ 1,\dots ,n
\right \} \right \} .
\end{equation*}
Note that $F$ can be identified with
\begin{align*}
F & \simeq H_{0}^{1}\left (\left (d_{0},d_{1}\right )\right )\times H_{0}^{1}
\left (\left (d_{2},d_{3}\right )\right )\times \dots \times H_{0}^{1}
\left (\left (d_{n},d_{n+1}\right ).\right )
\end{align*}
Furthermore, let $T_{k,\nu}$ denote the operator on
$H_{0}^{1}\left (\left (d_{k},d_{k+1}\right )\right )$ with respect to
the form
$\left (f,g\right )\mapsto \int _{\left (d_{k},d_{k+1}\right )}fg\;
\mathrm d\nu $ and let $T_{F,\nu}$ be the operator on $F$ with respect
to the form $\left (f,g\right )\mapsto \int fg\;\mathrm d\nu $. We define
the eigenvalue counting function of
$S\in \left \{ T_{\text{$\nu$}},T_{k,\nu},T_{F,\nu}\right \} $ by
\begin{equation*}
N\left (x,S\right ) :=  \card \left \{ n\in \mathbb N:\lambda _{n}^{S}
\geq x\right \} ,\:x>0.
\end{equation*}
Then the following sub-/superadditivity holds true.
%
\begin{prop}
\label{prop:Addiitivity}
For all $x\geq 0$, we have
\begin{equation*}
\sum _{k=0}^{n}N\left (x,T_{k,\nu}\right )=N\left (x,T_{F,\nu}\right )
\leq N\left (x,T_{\nu}\right )\leq \sum _{k=0}^{n}N\left (x,T_{k,\nu}
\right )+n.
\end{equation*}
\end{prop}

\begin{proof}
From max-min principle we deduce
\begin{equation*}
N\left (x,T_{F,\nu}\right )\leq N\left (x,T_{\nu}\right ).
\end{equation*}
Moreover, we have $\dim \left (H_{0}^{1}/F\right )=n$. Hence, it follows
from \citep[Proposition 1]{MR1328700}
\begin{equation*}
N\left (x,T_{\nu}\right )\leq N\left (x,T_{F,\nu}\right )+n.
\end{equation*}
It remains to show
$N\left (x,T_{F,\nu}\right )=\sum _{k=0}^{n}N\left (x,T_{k,\nu}
\right )$. Let $f$ be an eigenfunction with eigenvalue $\lambda >0$ of
$T_{k,\nu}$ with $\nu \left (\left (d_{k},d_{k+1}\right )\right )>0$. Then
define
\begin{equation*}
g(x)=\mathbbm{1}_{(d_{k},d_{k+1})}f,
\end{equation*}
it follows $g\in F$ and we have for all $h\in F$
\begin{equation*}
\lambda \int _{(0,1)}\nabla h\nabla g\;\mathrm d\Lambda =\lambda
\int _{(d_{k},d_{k+1})}\nabla h\nabla g\;\mathrm d\Lambda =\int _{(d_{k},d_{k+1})}hg
\;\mathrm d\nu =\int _{(0,1)}hg\;\mathrm d\nu .
\end{equation*}
This implies
$N\left (x,T_{F,\nu}\right )\geq \sum _{k=0}^{n}N\left (x,T_{k,\nu}
\right )$. On the other hand, if $f$ is an eigenfunction with eigenvalue
$\lambda >0$ of $T_{F,\nu}$, then $g=\mathbbm{1}_{(d_{k},d_{k+1})}f$ is
an eigenfunction with eigenvalue $\lambda $ of $T_{k,\nu}$ provided
$g\neq 0$ and $\nu \left (\left (d_{k},d_{k+1}\right )\right )>0$. Hence,
we obtain
$N\left (x,T_{F,\nu}\right )=\sum _{k=0}^{n}N\left (x,T_{k,\nu}
\right )$.
\end{proof}



\section*{Acknowledgment}
This research was supported by the {DFG} grant  {Ke 1440/3-1}. We would like
to thank the anonymous referee for her/his valuable comments, which have
contributed to a significant improvement of the presentation.

\phantomsection

\end{document}